\documentclass[a4paper]{amsart}

\usepackage{amsmath}
\usepackage{amsthm}
\usepackage{amssymb}
\usepackage{xspace}
\usepackage{enumerate}
\usepackage{mathtools}

\usepackage{color}

\theoremstyle{plain}
\newtheorem{theorem}{Theorem}
\newtheorem{lemma}[theorem]{Lemma}

\newcommand{\rea}{\ensuremath{\mathbf{R}}\xspace}
\newcommand{\nat}{\ensuremath{\mathbf{N}}\xspace}

\newcommand{\intd}[1]{\,\mathrm{d}#1 \,}
\newcommand{\hmeas}{\ensuremath{\mathcal{H}}\xspace}
\newcommand{\seq}[1]{\ensuremath{\underline{#1}}\xspace}
\newcommand{\ourP}[1]{\ensuremath{\mathbf{P}_{\!#1}}\xspace}
\newcommand{\overbar}[1]{{\mkern 3mu\overline{\mkern-3mu#1\mkern-1mu}\mkern 1mu}}
\newcommand{\oball}[2]{\ensuremath{B{\left(#1, #2\right)}}\xspace}
\newcommand{\cball}[2]{\ensuremath{\overbar{B}{\left(#1, #2\right)}}\xspace}
\newcommand{\crect}[2]{\ensuremath{\overbar R \left(#1, #2\right)}\xspace}

\newcommand\restr[2]{{
  \left.\kern-\nulldelimiterspace
  #1
  \vphantom{|}
  \right|_{#2}
  }}

\DeclareMathOperator{\dimh}{dim_H}
\DeclareMathOperator*{\E}{\mathbf{E}}

\begin{document}

\title[Limsup sets of random rectangles]%
{Hausdorff dimension of limsup sets of random rectangles in
products of regular spaces}

\author[F. Ekstr\"om]{Fredrik Ekstr\"om$^1$}
\address{Department of Mathematical Sciences, P.O. Box 3000,
         90014 University of Oulu, Finland$^{1,2,3,4}$}
\email{fredrik.ekstrom@oulu.fi$^1$}

\author[E. J\"arvenp\"a\"a]{Esa J\"arvenp\"a\"a$^2$}
\email{esa.jarvenpaa@oulu.fi$^2$}

\author[M. J\"arvenp\"a\"a]{Maarit J\"arvenp\"a\"a$^3$}
\email{maarit.jarvenpaa@oulu.fi$^3$}

\author[V. Suomala]{Ville Suomala$^4$}
\email{ville.suomala@oulu.fi$^4$}

\thanks{We acknowledge the support of the Centre of Excellence in Analysis and
Dynamics, funded by the Academy of Finland. We thank P.~Mattila,
S.~Seuret, P.~Shmerkin and J.~Tyson for useful discussions, and
the program Fractal Geometry and Dynamics, held at Institut Mittag--Leffler.}


\begin{abstract}
The almost sure Hausdorff dimension of the limsup set of randomly
distributed rect\-angles in a product of Ahlfors regular metric spaces is
computed in terms of the singular value function of the rectangles.
\end{abstract}

\maketitle

\section{Introduction}\label{intro}

Limsup sets appear in various fields of mathematics. They can be traced
back to the prominent works of Borel \cite{Borel97} and Cantelli
\cite{Cantelli17} leading to the Borel-Cantelli lemma.
Limsup sets are encountered, for example, in the study of Besicovitch-Eggleston
sets
concerning $k$-adic expansions of real numbers
\cite{Besicovitch34,Eggleston49} as well as in the
classical results of Khintchine \cite{Khintchine26} and Jarn\'\i k
\cite{Jarnik32} concerning well-approximable numbers, along with other questions
related to  Diophantine approximation. Recently, a considerable amount of
attention has been paid to dimensional properties of random and dynamically
defined limsup sets which are a class of limsup sets.
Given a space $X$ and a sequence $(A_n)$ of subsets of $X$, the limsup set
$E$ consists of those points in $X$ which are covered by infinitely many sets
$A_n$, that is, 
\[
E=\limsup_nA_n=\bigcap_{n=1}^\infty\bigcup_{k=n}^\infty A_k.
\]
If the space $X$ is a group or the sets $A_n$ are balls,
the limsup set can be made random by translating the sets $A_n$.
In the literature there are two main lines of randomness:~ the sets are
translated according to some
probability measure or the translations are determined by a typical orbit of
a measure preserving dynamical system.

The study of dimensional properties of random limsup sets was initiated by Fan
and Wu \cite{FanWu04}. They derived the following formula for the Hausdorff
dimension of a typical random limsup set generated by arcs that are randomly
placed on the circle
according to the Lebesgue measure:
\begin{equation}\label{dimarcs}
\dimh E=\inf\left\{t;\, \sum_n l_n^t<\infty\right\}\wedge 1,
\end{equation}
where $l_n$ is the length of the arc $A_n$ and $x\wedge y$ is the minimum of
two numbers $x$ and $y$. Durand \cite{Durand10} generalised this result
by proving a dichotomy result on Hausdorff measures defined in terms of
general gauge functions. His result implies that, replacing lengths of
arcs by radii of balls, formula \eqref{dimarcs} is valid for random
limsup sets generated by balls which are uniformly distributed on
the $d$-dimensional torus.
As shown in \cite{JJKLS14}, in the case where generating sets are balls
a shorter proof of \eqref{dimarcs}, which also
extends to Ahlfors regular metric spaces, can be obtained using the mass
transference principle proved by Beresnevich and Velani
\cite{BeresnevichVelani06}.
J\"arvenp\"a\"a et al.~ \cite{JJKLS14} derived a
formula for the almost sure Hausdorff
dimension of random limsup sets induced by rectangle-like sets uniformly
distributed on the $d$-dimensional torus. The formula is similar to
\eqref{dimarcs}\,---\,only the arc lengths are replaced by the
singular value functions of rectangles
(for a modification of this definition see \eqref{singularvalue}).
Persson \cite{Persson15} proved a lower bound for the Hausdorff dimension which
is valid for random limsup sets generated by general open sets. Finally,
in \cite{FJJSUP}, Feng et al.~ developed a general dimension theory
for random limsup sets. In \cite{FJJSUP} the results are valid in Riemann
manifolds, generating sets are assumed to be only Lebesgue measurable
satisfying a necessary density assumption and any distribution will do as long
as it is not purely singular with respect to the Riemann volume.

Dimensional properties of random limsup sets driven by singular measures have
so far been studied in the case where the generating sets are balls in
Euclidean spaces or in the shift space. In
\cite{FanSchmelingTroubetzkoy13}, Fan et al.~ focused on dynamically defined
limsup sets in the circle, where
the balls are distributed along typical orbits under the angle doubling map
with respect to a Gibbs measure. Seuret \cite{SeuretUP} considered the
case where balls are placed
randomly according to a Gibbs measure. Both instances lead to the same
conclusion: formula \eqref{dimarcs} has to be modified by taking into
account the multifractal spectrum of the Gibbs measure. Ekstr\"om and Persson
studied in \cite{EkstromPerssonUP} limsup sets generated by balls randomly
distributed according to a general measure. They derived upper
and lower bounds for the Hausdorff dimension of a typical random limsup set.
These bounds agree and give a formula for the Hausdorff dimension for several
classes of measures.

In this paper, we initiate the study of dimensional properties of random
limsup sets generated by more general sets than balls and distributed
according to a singular measure. Compared to the cases where the generating
sets are balls or the measure is non-singular, the situation turns out to be
subtler. For example, consider in the plane a limsup set generated by
rectangles $R_n=[-r_{n,1},r_{n,1}]\times [-r_{n,2},r_{n,2}]$ such that
$r_{n,1}\le r_{n,2}$ for all $n$. Suppose that the driving measure is the
1-dimensional Hausdorff measure restricted to a unit line segment making an
angle $\alpha\in [0,\frac\pi 2]$ with the $x$-axis. One easily sees that for
$\alpha<\frac\pi 2$, the dimension of the limsup set is determined by the
sequence $(r_{n,1})$ while for $\alpha=\frac\pi 2$ only the sequence $(r_{n,2})$
plays a role. Thus there is no dimension formula involving
only the shapes of the generating sets and the multifractal spectrum of the
measure. Therefore, more refined information is needed. We will concentrate
on random limsup sets generated by rectangles in products of Ahlfors regular
metric spaces distributed according to a regular measure. Our main theorem
(Theorem~\ref{thm:main}) gives the almost sure Hausdorff
dimension of random limsup sets in this setting. We also illustrate
(Theorem~\ref{convthm}) how our method can be applied to give a new simpler
proof of the main theorem in \cite{JJKLS14}.

The paper is organised as follows. The main theorem is stated and an outline of
the proof is given in Section~\ref{overview}. In Section~\ref{regularspace} we
prove auxiliary results needed when verifying the upper bound of the dimension
in Section~\ref{upperbound}. Section~\ref{auxiliary} is devoted to 
preliminary results aiming at the completion of the main theorem by
establishing a lower bound for the dimension in Section~\ref{mainproofsection}.
Finally, in Section~\ref{altsection} we give a further application of our method.

\section{Main result}\label{overview}

Let $X_1, \ldots, X_d$ be metric spaces and consider the product space
$X_1^d = \bigtimes_{i = 1}^d X_i$ with the metric
	\[
	d(x, \widetilde x) = \max_{1 \leq i \leq d} d_i(x_i, \widetilde x_i).
	\]
The closed \emph{rectangle} in $X_1^d$ with centre $x = (x_1, \ldots, x_d)$
and ``side radii'' $r = (r_1, \ldots, r_d)$ is the set
	\[
	\crect{x}{r} = \bigtimes_{i = 1}^d \cball{x_i}{r_i},
	\]
where $\cball{x_i}{r_i}$ is the closed ball in $X_i$ with radius $r_i$ centred
at $x_i$. If $\seq x = (x_n)$ is a sequence of points
$x_n = (x_{n, 1}, \ldots, x_{n, d})$
in $X_1^d$ and $\seq r = (r_n)$ is a sequence of $d$-tuples
$r_n = (r_{n, 1}, \ldots, r_{n, d})$ of positive numbers, define
	\[
	E_{\seq r}(\seq x) = \limsup_n \crect{x_n}{r_n}.
        \]

Next, let $\mu$ be a Borel probability measure on $X_1^d$ and define
the probability space $(\Omega, \ourP\mu)$ by $\Omega = (X_1^d)^\nat$ and
$\ourP\mu = \mu^\nat$. A sample $\omega = (\omega_n)$ from this probability space
represents a sequence of points $\omega_n = (\omega_{n, 1},\ldots, \omega_{n, d})$
in $X_1^d$, chosen independently according to $\mu$. The Hausdorff dimension of
$E_{\seq r}$ is almost surely constant with respect to $\ourP\mu$, since the event
that it is less than any given value is a tail event. (For the
proof of measurability of this event, see \cite[Lemma 3.1]{JJKLSX17}.) Let
	\[
	f_\mu(\seq r) = \text{$\ourP\mu$-a.s.~value of } \dimh E_{\seq r}.
	\]

For $r = (r_1, \ldots, r_d)$ and $s = (s_1, \ldots, s_d)$ such that
$r_i > 0$ and $s_i \geq 0$ for every $i$, let $\tau$ be
a permutation of $\{1, \ldots, d\}$ such that
$r_{\tau(1)} \geq \ldots \geq r_{\tau(d)}$ and define the
\emph{singular value function} $\Phi_r^s$ on
$[0, s_1 + \ldots + s_d]$ by
	\begin{equation}\label{singularvalue}
        \Phi_r^s(t) = r_{\tau(i)}^{t - \sum_{j = 1}^{i - 1} s_{\tau(j)}}
	\cdot \prod_{j = 1}^{i - 1}r_{\tau(j)}^{s_{\tau(j)}}
	\,\text{ for }
	t \in [s_{\tau(1)} + \ldots + s_{\tau(i - 1)},
	s_{\tau(1)} + \ldots + s_{\tau(i)}].
        \end{equation}
The product $\Phi_r^s(t)$ is formed by distributing a ``total exponent''
$t$ among the bases $r_1, \ldots, r_d$, giving an exponent of at most $s_i$
to the base $r_i$, and prioritising larger bases over smaller ones.

A Borel measure $\mu$ on a metric space $X$ is \emph{$s$-regular} if there
is a constant $c \geq 1$ such that
$c^{-1} r^s \leq\mu\left(\oball{x}{r}\right)\leq c r^s$
for every $x$ in $X$ and $r$ in $[0, 2|X|]$.

\begin{theorem}\label{thm:main}
For $i = 1, \ldots, d$, let $X_i$ be a compact metric space and let
$\mu_i$ be an $s_i$-regular Borel probability measure on $X_i$.
Let $\mu = \bigtimes_{i = 1}^d \mu_i$. Then
	\[
	f_\mu(\seq r) = \inf\left\{
	t; \, \sum_n \Phi_{r_n}^s(t) < \infty
	\right\} \wedge \left(
	s_1 + \ldots + s_d
	\right).
	\]
\end{theorem}

In case $r_{n, i}$ is decreasing in both $n$ and $i$,
and there is a positive number $\alpha$ such that
$r_{n, i} \leq n^{-\alpha}$ for every $i$ for all but finitely many
$n$, a bit of computation shows that the expression for $f_\mu(\seq r)$
in the theorem equals
	\[
	\min_{1 \leq i \leq d} \limsup_{n \to \infty}
	\left(
	\frac{\log n}{-\log r_{n, i}}  +
	\sum_{j = 1}^{i - 1} s_j \left( 1 -
	\frac{\log r_{n, j}}{\log r_{n, i}} \right)
	\right).
	\]

Theorem~\ref{thm:main} can be applied also for limsup sets driven by purely
singular measures. For example, let $s,t\in [0,1]$ with $s\wedge t<1$ and let
$X,Y\subset [0,1]$ be compact sets such
that the restrictions of the $s$-dimensional and $t$-dimensional Hausdorff
measures to $X$ and $Y$ are $s$-regular and $t$-regular, respectively. Then
the product measure is singular with respect to the 2-dimensional Lebesgue
measure. Since $X\times Y\subset [0,1]^2$ is closed, the limsup set defined
in $[0,1]^2$ is included in $X\times Y$ provided $(r_n)$ tends to zero. Thus
Theorem~\ref{thm:main} can be employed.

\subsection*{Overview of the proof}
The upper bound for $f_\mu$ is obtained in a standard way by showing that
each rectangle $\crect{x_n}{r_n}$ has a cover by balls whose contribution
to the sum in the definition of the $t$-dimensional Hausdorff measure is
comparable to $\Phi_{r_n}^s(t)$. See Lemma~\ref{svfublemma} below.

The proof of the lower bound is done by induction on $d$.
For $d = 1$, the statement of Theorem~\ref{thm:main} reduces to the well-known
result that if $X$ is a compact metric space and $\mu$ is an $s$-regular Borel
probability measure on $X$, then for $\ourP\mu$-a.e.~$\omega$
	\[
	\dimh \left( \limsup_n \cball{\omega_n}{r_n} \right) =
	\inf \left\{t; \, \sum_n r_n^t < \infty\right\} \wedge s
	\]
(see \cite[Theorem~2.1]{JJKLSX17} and~\cite[Proposition~4.7]{JJKLS14}).
Proving the lower bound for $d \geq 2$, it may be assumed that every
$r_n$ is decreasing, that is,
that $r_{n, 1} \geq \ldots \geq r_{n, d}$, since there are only finitely
many ways in which a $d$-tuple $r = (r_1, \ldots, r_d)$ can be ordered
(see the proof of the theorem in Section~\ref{mainproofsection} for details).
Under this assumption, the induction step is as follows. Let
$\pi$ be the projection $X_1^d \to X_1^{d - 1}$, and if $a$ is something
$d$-dimensional (a tuple, a sequence of tuples or a
product measure), let $a'$ be the image of $a$ under $\pi$.
Then for \ourP\mu-a.e.~$\omega$,
	\[
	\dimh E_{\seq r}(\omega) \geq
	\dimh \pi \left( E_{\seq r}(\omega) \right) \geq
	\dimh E_{\seq r'}(\omega')
	\]
(see Lemma~\ref{projectionlemma} below), and $\Phi_r^s(t) = \Phi_{r'}^{s'}(t)$
for $t \in [0, s_1 + \ldots + s_{d - 1}]$. Thus if
	\[
	\sum_n \Phi_{r_n}^s (s_1 + \ldots + s_{d - 1}) < \infty
	\]
then the theorem for $d - 1$ immediately implies the theorem for $d$.

Otherwise, there is some $u \in [0, s_d]$ such that
	\[
	\sum_n \Phi_{r_n}^s (s_1 + \ldots + s_{d - 1} + u) = \infty.
	\]
From this it can be deduced, using the theorem for $d = 1$, that for any
fixed $x' \in X_1^{d - 1}$, the fibre
	\[
	E_{\seq r}^{x'}(\omega) = \left\{
	x_d \in X_d; \, (x', x_d) \in E_{\seq r}(\omega)
	\right\}
	\]
almost surely has Hausdorff dimension greater than or equal to $u$. By
Fubini's theorem it follows that for \ourP\mu-a.e.~$\omega$
	\[
	\mu'\left( \left\{ x'; \, \dimh E_{\seq r}^{x'}(\omega) \geq u \right\}
	\right)	= 1.
	\]
Since the support of $\mu'$ has dimension $s_1 + \ldots + s_{d - 1}$, this
implies (by Lemma~\ref{dimhlemma} below) that
	\[
	\dimh E_{\seq r}(\omega) \geq s_1 + \ldots + s_{d - 1} + u,
	\]
and then the theorem is proved by taking supremum over $u$.

\subsection*{Notation and conventions}
The open and closed balls with centre $x$ and radius $r$ are denoted by
$\oball{x}{r}$ and $\cball{x}{r}$, respectively. All rectangles that
appear are closed, and denoted as
$\crect{x}{r} = \bigtimes_{i = 1}^d \cball{x_i}{r_i}$,
where $x = (x_1, \ldots, x_d)$ and $r = (r_1, \ldots, r_d)$. All measures that
appear are Borel measures, but this will not be explicitly stated every
time\,---\,thus ``a measure'' should be read as ``a Borel measure.''

\section{Regular spaces}\label{regularspace}

A \emph{regular space} is a metric space $X$ together with a measure
$\mu$ on $X$ for which there exist a constant $c \geq 1$ and a non-negative
number $s$ such that for every $x \in X$ and $r \in [0, 2|X|]$,
	\[
	c^{-1} r^s \leq \mu\left( \oball{x}{r} \right) \leq c r^s.
	\]
When the exponent or the constant and the exponent need to be emphasised,
$X$ will be called $s$-regular or $(c, s)$-regular. The inequalities in
the definition are satisfied for all $r$ if and only if they are
satisfied  for all $r$ with $\cball{x}{r}$ in place of $\oball{x}{r}$.

A subset $A$ of a metric space is \emph{$r$-sparse} if the distance between
any two distinct points in $A$ is greater than or equal to $r$. If $A$
is a maximal $r$-sparse subset of $B$, then
$B \subset \bigcup_{x \in A} \cball{x}{r}$.

\begin{lemma} \label{sparselemma}
Let $X$ be a $(c, s)$-regular space and let $r$ and $R$ be positive numbers
such that $r \leq 2R$ and $R \leq |X|$. Let $x_0 \in X$. Then there exists a
maximal $r$-sparse subset $A$ of $\cball{x_0}{R}$, and every
such $A$ satisfies
	\[
	c^{-2} \left(\frac{R}{r}\right)^s
	\leq \# A \leq
	4^s c^2 \left(\frac{R}{r} \right)^s.	
	\]

\begin{proof}
First consider \emph{any} $r$-sparse subset $A$ of $\cball{x_0}{R}$.
Then since the balls $\{ \oball{x}{r / 2} \}_{x \in A}$ are disjoint
and included in $\cball{x_0}{2R}$,
	\[
	\#A \cdot c^{-1} (r / 2)^s \leq
	\sum_{x \in A} \mu\left(\oball{x}{r / 2}\right)
	\leq \mu\left(\cball{x_0}{2R}\right) \leq c \left(2R\right)^s.
	\]
The upper bound for $\# A$ follows. It also follows that starting
with $A = \emptyset$ and repeatedly adding points from $\cball{x_0}{R}$
such that $A$ is always $r$-sparse, the process must end after
a bounded number of steps. At that point, a maximal $r$-sparse subset
of $\cball{x_0}{R}$ is obtained. Now, assuming that $A$ is maximal,
	\[
	c^{-1} R^s \leq \mu\left(\oball{x_0}{R}\right) \leq \sum_{x \in A}
	\mu\left( \cball{x}{r} \right) \leq \# A \cdot cr^s,
	\]
from which the lower bound for $\#A$ follows.
\end{proof}
\end{lemma}

\begin{lemma} \label{ballcoverlemma}
Let $X$ be a $(c, s)$-regular space and let $r$ and $R$ be positive
numbers such that $r \leq 2R$. Let $B$ be a closed
ball  of radius $R$ in $X$. Then $B$ can be covered by $4^s c^2 (R / r)^s$
closed balls of radius $r$.

\begin{proof}
It may be assumed that $R \leq |X|$.
Let $A$ be a maximal $r$-sparse subset of $B$. Then
$B \subset \bigcup_{x \in A} \cball{x}{r}$, and
$\# A \leq 4^s c^2 (R / r)^s$ by Lemma~\ref{sparselemma}.
\end{proof}
\end{lemma}

A \emph{cube} in $X_1^d$ is a rectangle whose sides are balls of the
same radius\,---\,this is the same as a ball in $X_1^d$.

\begin{lemma} \label{rectanglecoverlemma}
For $i = 1, \ldots, d$ let $X_i$ be a $(c_i, s_i)$-regular space,
and consider a rect\-angle $R = \bigtimes_{i = 1}^d \cball{x_i}{r_i}$
in $X_1^d$. Let $r > 0$ and let $I = \{ i; \, r_i > r \}$.
Then $R$ can be covered by $M$ cubes of radius $r$, where
	\[
	M = \prod_{i \in I} 4^{s_i} c_i^2 \left(\frac{r_i}{r}\right)^{s_i}
	\]
with the interpretation $M=1$ if $I=\emptyset$.
          
\begin{proof}
By Lemma~\ref{ballcoverlemma}, there is for each $i \in I$ a subset
$A_i$ of $X_i$ such that
	\[
	\cball{x_i}{r_i} \subset \bigcup_{a_i \in A_i} \cball{a_i}{r}
	\quad\text{and}\quad
	\# A_i \leq 4^{s_i} c_i^2 \left(\frac{r_i}{r}\right)^{s_i}.
	\]
For $i \notin I$, let $A_i = \{ x_i \}$. Then $R$ is covered by the
set of cubes of the form $\bigtimes_{i = 1}^d \cball{a_i}{r}$,
where $a_i \in A_i$ for each $i$,
and the number of such cubes is less than or equal to $M$.
\end{proof}
\end{lemma}

\section{Upper bound}\label{upperbound}

The following lemma gives the upper bound for $f_\mu$ in Theorem~\ref{thm:main}.

\begin{lemma} \label{svfublemma}
For $i = 1, \ldots, d$ let $X_i$ be a $(c_i, s_i)$-regular space, and
let $\seq x = (x_n)$ be a sequence of points in $X_1^d$. Let $\seq r = (r_n)$ be a
sequence of $d$-tuples of positive numbers. Then 
	\[
	\dimh E_{\seq r}(\seq x)
	\leq \inf \left\{
	t; \, \sum_n \Phi_{r_n}^s(t) < \infty	
	\right\} \wedge (s_1 + \ldots + s_d).
	\]

\begin{proof}
It is clear that $\dimh E_{\seq r}(\seq x) \leq \dimh X_1^d = s_1 + \ldots + s_d$.
Suppose that $t < s_1 + \ldots + s_d$ is such that
$\sum_n \Phi_{r_n}^s(t) < \infty$. Let $\delta > 0$.
Then there is some $n_0 = n_0(\delta)$ such that $r_{n, i} \leq \delta$
for every $i$ and every $n \geq n_0$, and the limsup set is included in
$\bigcup_{n \geq n_0} \crect{x_n}{r_n}$.

For each $n$, let $\tau_n$ be a permutation of $\{1, \ldots, d\}$
such that $r_{n, \tau_n(1)} \geq \ldots \geq r_{n, \tau_n(d)}$ and
let $i_n$ be such that
$t \in [s_{\tau_n(1)} + \ldots + s_{\tau_n(i_n - 1)}, \,
s_{\tau_n(1)} + \ldots + s_{\tau_n(i_n)}]$.
By Lemma~\ref{rectanglecoverlemma} there is a cover of $\crect{x_n}{r_n}$
by $M_n$ cubes of radius $r_{n, \tau_n(i_n)}$, where
	\[
	M_n \leq C \prod_{j = 1}^{i_n - 1}
	\left( \frac{r_{n, \tau_n(j)}}{r_{n, \tau_n(i_n)}} \right)^{s_{\tau_n(j)}}
	\]
for some constant $C$ that is independent of $n$. Note that if
$r_{n, \tau_n(j)}=r_{n, \tau_n(i_n)}$ for some $j<i_n$, Lemma~\ref{rectanglecoverlemma}
gives an upper bound where the product is up to the largest $j_0$ such that
$r_{n, \tau_n(j_0)}>r_{n, \tau_n(i_n)}$. The product can be extended up to $i_n-1$ by
noting that the extra terms added to the product are equal to 1. Thus
	\[
	\hmeas_\delta^t (E_{\seq r}(\seq x)) \leq
	2^t \sum_{n = n_0}^\infty M_n r_{n, \tau_n(i_n)}^t
	\leq
	2^t C \sum_{n = 1}^\infty \Phi_{r_n}^s(t).
	\]
Since the last expression is finite and independent of $\delta$
the limsup set has finite $\hmeas^t$-measure, and hence Hausdorff
dimension less than or equal to $t$.
\end{proof}
\end{lemma}

\section{Auxiliary lemmas}\label{auxiliary}
  
The following standard lemma
will be used in the proof of Theorem~\ref{thm:main}.

\begin{lemma} \label{asdivsumlemma}
Let $(\xi_n)$ be a sequence of independent random variables taking values in
$[0, 1]$, such that $\sum_n \E \xi_n = \infty$. Then almost surely,
$\sum_n \xi_n = \infty$.

\begin{proof}
If $M \leq \frac{1}{2} \sum_{n = 1}^N \E \xi_n$ then, using Markov's inequality
and the assumption that the variables $\xi_n$ are independent and take values
in $[0, 1]$,
	\begin{align*}
	\mathbf P \left\{ \sum_{n = 1}^N \xi_n \leq M \right\} &=
	\mathbf P \left\{ \sum_{n = 1}^N \left(\E \xi_n - \xi_n\right) \geq
	\sum_{n = 1}^N \E \xi_n - M \right\} \\
	&\leq
	\mathbf P \left\{ \left(\sum_{n = 1}^N\left(\E \xi_n - \xi_n\right)
        \right)^2 \geq \left(\sum_{n = 1}^N \E \xi_n - M\right)^2 \right\} \\
	&\leq
	\frac{\E \left(\sum_{n = 1}^N \left( \E \xi_n - \xi_n\right)\right)^2}
	{\left(\sum_{n = 1}^N \E \xi_n - M\right)^2}
	=
	\frac{\sum_{n = 1}^N \left(\E \xi_n^2 - \left(\E \xi_n\right)^2\right)}
	{\left(\sum_{n = 1}^N \E \xi_n - M\right)^2} \\
	&\leq
	\frac{4 \sum_{n = 1}^N \E \xi_n^2}
	{\left(\sum_{n = 1}^N \E \xi_n \right)^2}
	\leq
	\frac{4 \sum_{n = 1}^N \E \xi_n}
	{\left(\sum_{n = 1}^N \E \xi_n \right)^2}
	=
	\frac{4}{\sum_{n = 1}^N \E \xi_n}
        \le\frac 2M.     
	\end{align*}
Thus
	\begin{align*}
	\mathbf P \left\{ \sum_{n = 1}^\infty \xi_n < \infty \right\}
	&=
	\mathbf P \left(
	\bigcup_{M = 1}^\infty \bigcap_{N = 1}^\infty
	\left\{ \sum_{n = 1}^N \xi_n \leq M \right\}
	\right) \\
	&=
	\lim_{M \to \infty} \lim_{N \to \infty}
	\mathbf P \left\{ \sum_{n = 1}^N \xi_n \leq M \right\}
	= 0. \qedhere
	\end{align*}
\end{proof}
\end{lemma}

The following lemma is a consequence of Corollary~2.10.27 in
Federer's book~\cite{Federer69}, using that a complete regular
space is boundedly compact (meaning that every closed bounded
subset of the space is compact, or equivalently, that the space is
complete and every bounded subset is totally bounded).

\begin{lemma}[{\cite[Corollary~2.10.27]{Federer69}}] \label{dimhlemma}
Let $X$ and $Y$ be complete regular metric spaces and let $A$ be a subset of
$X \times Y$ such that
	\[
	\dimh \left\{x \in X; \, \dimh A^x \geq \beta \right\} \geq \alpha,
	\]
where
	\[
	A^x = \left\{y \in Y; \, (x, y) \in A \right\}.
	\]
Then $\dimh A \geq \alpha + \beta$.
\end{lemma}

According to the following lemma, typical limsup sets are dense.

\begin{lemma} \label{nonemptylemma}
Let $X$ be a complete separable metric space, let $\mu$ be a fully supported
probability measure on $X$ and let $\seq r$ be a sequence of positive numbers.
Then for \ourP\mu-a.e.~$\omega$, the set
$E_{\seq r}(\omega)=\limsup_n\cball{\omega_n}{r_n}$ is dense in $X$. In
particular, $E_{\seq r}(\omega)\ne\emptyset$.

\begin{proof}
If $B$ is an open ball in $X$ then $\mu(B) > 0$, and thus
$\# \{n; \, \omega_n \in B\} = \infty$ for \ourP\mu-a.e.~$\omega$
by the Borel--Cantelli lemma (or Lemma~\ref{asdivsumlemma}).
Since $X$ is separable there is a
basis for the topology on $X$ consisting of a countable family of
open balls. It follows that \ourP\mu-a.e.~$\omega$ is such that
$\# \{n; \, \omega_n \in B\} = \infty$ for every open ball $B$.
Fix such an $\omega$.
Let $B_0$ be an open ball in $X$ and let $n_0 = 0$.
Assuming that the open ball $B_{k - 1}$ and the natural number $n_{k - 1}$
have been defined, let $n_k > n_{k - 1}$ be such that
$\omega_{n_k} \in B_{k - 1}$ and let $B_k$ be an open ball such that
$\overbar B_k \subset B_{k - 1} \cap \oball{\omega_{n_k}}{r_{n_k}}$.
Then
	\[
	\emptyset \neq \bigcap_{k = 1}^\infty \overbar B_k
	\subset
	B_0 \cap \bigcap_{k = 1}^\infty \oball{\omega_{n_k}}{r_{n_k}}
	\subset
	B_0 \cap E_{\seq r}(\omega).
	\]
Since $B_0$ is arbitrary, this shows that $E_{\seq r}(\omega)$ is
dense.
\end{proof}
\end{lemma}

\section{Lower bound} \label{mainproofsection}

This section involves several $d$-dimensional objects,
whose $(d - 1)$-dimensional counterparts will be denoted by appending
a prime. Thus if $a$ is something $d$-dimensional, then $a' = \pi a$,
where $\pi$ is the appropriate projection to the first $d - 1$ coordinates.
For example, if
	\[
	s = (s_1, \ldots, s_d), \qquad
	\seq r = \big( (r_{n, 1}, \ldots, r_{n, d}) \big)_{n \in \nat}, \qquad
	\mu = \mu_1 \times \ldots \times \mu_d,
	\]
then
	\[
	s' = (s_1, \ldots, s_{d - 1}), \qquad
	\seq r' = \big( (r_{n, 1}, \ldots, r_{n, d - 1})\big)_{n \in \nat}, \qquad
	\mu' = \mu_1 \times \ldots \times \mu_{d - 1}.
	\]

According to the next lemma, the dimension $f_\mu(\seq r)$ of a typical limsup
set in $d$-dimensional product space is bounded from below by the corresponding
number $f_{\mu'}(\seq r')$ in $(d-1)$-dimensional product space. 

\begin{lemma} \label{projectionlemma}
Let $X_i$ be a complete separable metric space and $\mu_i$ be a probability
measure on $X_i$ for $i = 1, \ldots, d$. Let $\seq r = (r_n)$ be a
sequence of $d$-tuples $r_n = (r_{n, 1}, \ldots, r_{n, d})$. Set
$\mu = \bigtimes_{i = 1}^d \mu_i$.  Then
$f_\mu(\seq r) \geq f_{\mu'}(\seq r')$.

\begin{proof}
Let
	\[
	G = \left\{
	\omega; \, \dimh E_{\seq r}(\omega) \geq f_{\mu'}(\seq r')
	\right\},
	\]
and for each $\omega' \in \Omega' = \left(X_1^{d - 1}\right)^\nat$ let
	\[
	G^{\omega'} = \left\{ \sigma \in \left(X_d\right)^\nat; \,
	(\omega', \sigma) \in G \right\}.
	\]
Since $\ourP\mu$ can be identified with
$\ourP{\mu'} \times \ourP{\mu_d}$ and $G$ is universally measurable
\cite[Lemma 3.1]{JJKLS14}, Fubini's theorem gives
	\[
	\ourP\mu(G) =
	\int \ourP{\mu_d}\big(G^{\omega'}\big) \intd{\ourP{\mu'}}(\omega').
	\]
Thus it suffices to show that for every $\omega'$, the set
of $\sigma$ such that
	\[
	\dimh E_{\seq r}(\omega', \sigma) \geq \dimh E_{\seq r'}(\omega')
	\]
has full \ourP{\mu_d}-measure.

So fix $\omega'$ and let $s < \dimh E_{\seq r'}(\omega')$. Then
by~\cite[Theorems 48 and~57]{Rogers70} there is a compact subset
$K$ of $E_{\seq r'}(\omega')$ such that the measure
$\theta = \restr{\hmeas^s}{K}$ is positive and finite.
Let $\pi$ be the projection $X_1^d \to X_1^{d - 1}$ and for 
$\xi \in X_1^{d - 1}$ let $(n_k(\xi))$ be the increasing enumeration
of $\{n ; \, \xi \in \crect{\omega_n'}{r_n'}\}$. (Note that this set
is infinite for every $\xi\in E_{\seq r'}(\omega')$ and thus for
$\theta$-a.e.~$\xi$.) Then
	\begin{align*}
	\int \theta\big(\pi (E_{\seq r}(\omega', \sigma))\big)
	\intd{\ourP{\mu_d}}(\sigma)
	&=
	\iint \chi_{\pi (E_{\seq r}(\omega', \sigma))}(\xi)
	\intd\theta(\xi) \intd{\ourP{\mu_d}}(\sigma) \\
	&=
	\iint \chi_{\pi (E_{\seq r}(\omega', \sigma))}(\xi)
	\intd{\ourP{\mu_d}}(\sigma) \intd\theta(\xi) \\
	&=
	\int \ourP{\mu_d}\left\{ \limsup_k \cball{\sigma_{n_k(\xi)}}{r_{n_k(\xi),d}}
	\neq \emptyset \right\}
	\intd\theta(\xi) \\
        &=
        \theta(K),
	\end{align*}
using Lemma~\ref{nonemptylemma} in the last step. Thus for
\ourP{\mu_d}-a.e.~$\sigma$ the set $\pi (E_{\seq r}(\omega', \sigma))$ has full
$\theta$-measure, and in particular positive $\hmeas^s$-measure.
For such $\sigma$,
	\[
	\dimh E_{\seq r}(\omega', \sigma) \geq
	\dimh \pi (E_{\seq r}(\omega', \sigma)) \geq s.
	\]
Letting $s \to \dimh E_{\seq r'}(\omega')$ along a countable set concludes
the proof.
\end{proof}
\end{lemma}

\begin{proof}[Proof of Theorem~\ref{thm:main}]
The upper bound for $f_\mu(\seq r)$ claimed by the theorem follows from
Lemma~\ref{svfublemma}, so it only remains to prove the lower bound. Assume
first that 
        \begin{equation}\label{firstcase}
        r_{n, 1} \geq \ldots \geq r_{n, d}\qquad\text{for every }n;
        \end{equation}
then the proof is by induction on $d$. The theorem is known to hold for
$d = 1$, see \cite[Theorem~2.1]{JJKLSX17} and~\cite[Proposition~4.7]{JJKLS14}.
For $d \geq 2$ there are two cases.

\emph{Case 1.} Suppose that
$\sum_n \Phi_{r_n'}^{s'}(s_1 + \ldots + s_{d - 1}) < \infty$.
Then
	\[
	f_\mu(\seq r) \geq
	f_{\mu'}(\seq r') =
	\inf\left\{
	t; \, \sum_n \Phi_{r_n'}^{s'}(t) < \infty \right\}
	=
	\inf\left\{
	t; \, \sum_n \Phi_{r_n}^s(t) < \infty \right\}.		
	\]
Here the inequality is by Lemma~\ref{projectionlemma}, the first equality
is by the induction hypothesis and the second equality holds since
$\Phi_{r_n'}^{s'}(t) = \Phi_{r_n}^s(t)$ for $t \in [0, s_1 + \ldots + s_{d - 1}]$
by the assumption \eqref{firstcase}.

\emph{Case 2.} Suppose that $u \in [0, s_d]$ is such that
$\sum_n \Phi_{r_n}^s(s_1 + \ldots + s_{d - 1} + u) = \infty$.
For $\omega \in \Omega$ and $x' \in X_1^{d - 1}$ let
	\[
	E_{\seq r}^{x'} (\omega) =
	\left\{x_d \in X_d; \, (x', x_d) \in E_{\seq r} (\omega)\right\}
	\]
and let $\big(n_k^{x'}(\omega)\big)$ be the increasing enumeration of
$\big\{n; \, x' \in \crect{\omega_n'}{r_n'} \big\}$. Note that this set is
infinite if $x'\in E_{\seq r'}(\omega')$. Then
	\[
	E_{\seq r}^{x'} (\omega) =
	\limsup_k \cball{\omega_{n_k^{x'}(\omega), d}}{r_{n_k^{x'}(\omega), d}}.
	\]
Summing the $u$:th powers of the radii of the balls in the last expression
yields
	\begin{equation} \label{radiisumeq} 
	\sum_k r_{n_k^{x'}(\omega), d}^u =
	\sum_n \chi_{\crect{\omega_n'}{r_n'}}(x') \, r_{n, d}^u =
	\sum_n \chi_{\crect{x'}{r_n'}}(\omega_n') \, r_{n, d}^u.
	\end{equation} 
The last sum is a sum of independent random variables, and the sum of
their expectations is bounded from below by
	\[
	\sum_n \left(\prod_{i = 1}^{d - 1}c_i^{-1} r_{n, i}^{s_i}\right)r_{n, d}^u
	=
	c\sum_n \Phi_{r_n}^s(s_1 + \ldots + s_{d - 1} + u) = \infty.
	\]
Lemma~\ref{asdivsumlemma} then implies that~\eqref{radiisumeq}
diverges for \ourP{\mu'}-a.e.~$\omega'$. For such $\omega'$, the inequality 
$\dimh E_{\seq r}^{x'} (\omega', \sigma) \geq u$ is valid for
\ourP{\mu_d}-a.e.~$\sigma$ by the theorem for $d = 1$.

Thus, for each fixed $x'$, Fubini's theorem implies
	\[
	\ourP\mu \left\{\dimh E_{\seq r}^{x'} \geq u \right\} =
	\ourP{\mu'} \times \ourP{\mu_d} \left\{
	\dimh E_{\seq r}^{x'} \geq u
	\right\} = 1.
	\]
By Fubini's theorem it follows that
	\[
	\ourP\mu \times \mu'
	\left\{ (\omega, x'); \,
	\dimh E_{\seq r}^{x'}(\omega) \geq u	\right\} = 1,
	\]
and thus \ourP\mu-a.e.~$\omega$ is such that
	\[
	\mu'\left\{
	x'; \, \dimh E_{\seq r}^{x'}(\omega) \geq u	\right\} = 1.
	\]
For such $\omega$,
	\[
	\dimh E_{\seq r}(\omega) \geq s_1 + \ldots + s_{d - 1} + u
	\]
by Lemma~\ref{dimhlemma}, since $\mu'$ is
$(s_1 + \ldots + s_{d - 1})$-regular. Thus
	\begin{align*}
	f_\mu(\seq r) &\geq	s_1 + \ldots + s_{d - 1} +
	\sup \left\{u \in [0, s_d]; \,
	\sum_n \Phi_{r_n}^s(s_1 + \ldots + s_{d - 1} + u) = \infty \right\} \\
	&=
	\sup \left\{ t \leq s_1 + \ldots + s_d; \,
	\sum_n \Phi_{r_n}^s(t) = \infty \right\} \\
	&=
	\inf \left\{
	t; \, \sum_n \Phi_{r_n}^s(t) < \infty \right\} \wedge
	(s_1 + \ldots + s_d).
	\end{align*}

The theorem is now proved in the case when $r_n$ is decreasing for every $n$.
It follows that the theorem holds when all the $r_n$ are ordered in
the same way, that is, when there is a permutation $\tau$ of
$\{1, \ldots, d\}$ such that $r_{n, \tau(1)} \geq \ldots \geq r_{n, \tau(d)}$
for every $n$, since the singular value function and the Hausdorff dimension
do not change if the spaces $\{ (X_i, \mu_i) \}$ are relabelled.
In general, there is for each $n$ a permutation
$\tau_n$ such that $r_{n, \tau_n(1)} \geq \ldots \geq r_{n, \tau_n(d)}$.
Given a permutation $\tau$, let $N(\tau)$ be the set of $n$ such that
$\tau_n = \tau$ and let
$E_{\seq r}^\tau(\omega) = \limsup_{n \in N(\tau)} \crect{\omega_n}{r_n}$.
Then for \ourP\mu-a.e.~$\omega$,
	\begin{align*}
	\dimh E_{\seq r}(\omega) &=
	\dimh \left( \bigcup_\tau E_{\seq r}^\tau(\omega) \right)
	=
	\max_\tau \left( \dimh E_{\seq r}^\tau(\omega) \right) \\
	&=
	\max_\tau \left(
	\sup\left\{	t \leq s_1 + \ldots + s_d; \,
	\sum_{n \in N(\tau)} \Phi_{r_n}^s(t) = \infty
	\right\} \right) \\
	&=
	\sup\left( \bigcup_\tau \left\{ t \leq s_1 + \ldots + s_d; \,
	\sum_{n \in N(\tau)} \Phi_{r_n}^s(t) = \infty
	\right\} \right)\\
	&=
	\sup\left\{	t \leq s_1 + \ldots + s_d; \,
	\sum_n \Phi_{r_n}^s(t) = \infty
	\right\}. \qedhere
	\end{align*}
\end{proof}

\section{Convex bodies uniformly distributed in the unit cube} \label{altsection}

In this section it is shown how the method can be applied to give a new
shorter proof of the main result of \cite{JJKLS14}. Let $X_1^d$ be the unit cube
in $\rea^d$ and let $\mu$ be the uniform
probability measure on $X_1^d$. If $\seq K = (K_n)$ is a sequence of
convex bodies in $\rea^d$ containing $0$, let
	\[
	g(\seq K) = \ourP\mu\text{-a.s.~value of }
	\dimh\left(
	\limsup_n \big(\omega_n + K_n\big)
	\right).
	\]
By John's theorem~\cite{John48} there is a sequence $\seq E = (E_n)$
of ellipsoids such that $E_n \subset K_n \subset \widehat E_n$,
where $\widehat E_n$ is the ellipsoid obtained by dilating $E_n$ by a
factor of $d$ around its centre. 
Then $g(\seq E) \leq g(\seq K) \leq g\big(\seq{\widehat E}\big)$.
For each $n$ let $r = (r_1, \ldots, r_d)$ be the lengths of the
semiaxes of $E_n$ in decreasing order. Define $\mathbf 1 = (1, \ldots, 1)$.

\begin{theorem} \label{convthm}
In the setting above,
	\[
	g(\seq K) = \inf\left\{t; \, \sum_n \Phi_{r_n}^{\mathbf 1}(t) < \infty
	\right\} \wedge d.
	\]

\begin{proof}
It is clear that $g\big(\seq{\widehat E}\big) \leq d$, and a slight
reformulation of Lemma~\ref{svfublemma} implies that if
$\sum_n \Phi_{r_n}^{\mathbf 1}(t) < \infty$ then
$g\big(\seq{\widehat E}\big) \leq t$.  Conversely, if $t \leq d$ and
$\sum_n \Phi^{\mathbf 1}_{r_n}(t) = \infty$ then $g(\seq E) \geq t$.
This can be proved by induction on $d$. For $d = 1$ the statement
reduces to a known result as in Section~\ref{mainproofsection}.

Let $d \geq 2$. For each $n$ there is some $i = i(n)$ such
that $|v_n\cdot e_i| \geq |v_n| / \sqrt{d}$, where $v_n$ is the shortest
semiaxis of $E_n$. Let $N(i) = \{n; \, i(n) = i\}$. Then there must be some
$i$ such that $\sum_{n \in N(i)} \Phi^{\mathbf 1}_{r_n}(t) = \infty$.
By considering a subsequence of $\seq E$ and relabelling the coordinate
axes, it may thus be assumed that $|v_n\cdot e_d| \geq |v_n|/\sqrt{d}$
for every $n$. It follows that
$|\pi (w)| \geq |w| / \sqrt{d}$ for every vector $w$ in $\rea^d$
orthogonal to $v_n$, where $\pi$ is the projection
$X_1^d \to X_1^{d - 1}$. In particular, the lengths of the
semiaxes of $\pi(E_n)$ are comparable to $(r_{n, 1}, \ldots, r_{n, d - 1})$.
Now there are two cases, just as in Section~\ref{mainproofsection}.

\emph{Case 1.} If $t \leq d - 1$ then $g(\seq E) \geq t$ by the
induction hypothesis and an obvious modification of Lemma~\ref{projectionlemma}.

\emph{Case 2.} If $t = d - 1 + u$ where $u \in [0, 1]$, let
$E_n'$ be the ellipsoid in $X_1^{d - 1}$ obtained by dilating
$\pi(E_n)$ by a factor of $1 / 2$ around its centre. Then
the lengths of the semiaxes of $E_n'$ are comparable to
$r_n' = (r_{n, 1}, \ldots, r_{n, d - 1})$, and for
$x' \in E_n'$ the diameter of
$E_n^{x'} = \{x_d \in X_d; \, (x', x_d) \in E_n \}$
is comparable to $r_{n, d}$. Thus for every fixed $x' \in X_1^{d - 1}$,
	\[
	\sum_n \E l_n(\omega)^u \geq
	c\sum_n \Phi_{r_n}^{\mathbf 1}(t) = \infty,
	\]
where $l_n(\omega) = |(\omega_n + E_n)^{x'}|$ and $c$ is a constant. Using
Lemma~\ref{asdivsumlemma}, Fubini's theorem and Lemma~\ref{dimhlemma} as in
Section~\ref{mainproofsection} then shows that
	\[
	\dimh\left(	\limsup_n \big(\omega_n + E_n\big) \right)
	\geq d - 1 + u = t
	\]
for \ourP\mu-a.e.~$\omega$.
\end{proof}
\end{theorem}

\bibliographystyle{plain}
\bibliography{references}
\end{document}